\documentclass[12pt,reqno]{amsart}
\usepackage{amsmath,amssymb,amsfonts,amscd,latexsym,amsthm,mathrsfs,verbatim,comment}
\usepackage{enumerate}
\usepackage[unicode]{hyperref}
\newtheorem{theorem}{Theorem}
\newtheorem{prop}{Proposition}
\theoremstyle{remark}
\newtheorem{remark}{\bf Remark}
\newtheorem*{acknowledgements}{\bf Acknowledgements}
\newcommand{\ba}{\boldsymbol a}
\newcommand{\bt}{\boldsymbol t}
\newcommand{\bs}{\boldsymbol s}

\newcommand{\N}{\bold N}
\newcommand{\cA}{\mathcal A}
\newcommand{\cP}{\mathcal P}
\newcommand\Gal{\operatorname{Gal}}
\newcommand\Frob{\operatorname{Frob}}

\begin{document}

\title[Poor man's transcendence for Frobenius traces]{Poor man's transcendence\\for Frobenius traces of elliptic curves}

\author{Florian Luca}
\address{Stellenbosch University, Mathematics, Merriman Street 7600 Stellenbosch, South Africa; Max Planck Institute for Software Systems, Saarbr\"ucken, Germany; and Department of Computer Science, University of Oxford, UK}
\urladdr{https://florianluca.com/}

\author{Wadim Zudilin}
\address{Department of Mathematics, IMAPP, Radboud University, PO Box 9010, 6500~GL Nij\-me\-gen, Netherlands}
\urladdr{https://www.math.ru.nl/~wzudilin/}

\date{26 July 2025. \emph{Revised}: 3 February 2026}

\dedicatory{To Kaneko-sensei, in admiration and on the birthday occasion}


\maketitle

\begin{abstract}
Let $E$ be an elliptic curve without complex multiplication defined over $\mathbb Q$.
Viewing the sequence of its Frobenius traces $(a_p(E))_p$ indexed by primes $p$ as an element in the `poor man's ad\`ele ring', we prove its transcendence over~$\mathbb Q$.
\end{abstract}


Many reincarnations of multiple zeta values are known to this date; one of them was introduced recently by Kaneko and Zagier in \cite{KZ25} (see also \cite{Ka19,KMS25}).
This remarkable and innocently looking version called \emph{finite multiple zeta values} lives in the universe of `poor man's ad\`ele ring'
\[
\cA=\bigg(\prod_{p\in\cP}(\mathbb Z/p\mathbb Z)\bigg)\bigg/\bigg(\bigoplus_{p\in\cP}(\mathbb Z/p\mathbb Z)\bigg),
\quad\text{where}\; \cP=\{2,3,5,7,11,13,17,19,23,29,\dots\},
\]
introduced by Kontsevich in~\cite{Ko09} for somewhat different purposes; as the reviewer points out the origin of this ring goes back to the much earlier work of Ax \cite{Ax68}.
The elements of the ring $\cA$ are infinite vectors $\bt=(t_2,t_3,t_5,t_7,\dots)=(t_p)_{p\in\cP}$ indexed by prime numbers, with each $t_p$ viewed as a residue in $\mathbb Z/p\mathbb Z$.
The equivalence relation $\bt\sim\bs$ corresponds to $t_p\equiv s_p\bmod p$ for all sufficiently large primes $p$ (in particular, both $t_p\bmod p$ and $s_p\bmod p$ are well defined for such primes~$p$); a few first components of $\bt,\bs$ can remain undefined.
The field $\mathbb Q$ of rational numbers is naturally embedded in $\cA$ diagonally, so that $\cA$ is regarded as a $\mathbb Q$-algebra via this embedding.

Some very basic open questions about finite multiple zeta values include the nonvanishing of the latter, which in turn motivate investigating irrationality and even transcendence in the ring~$\cA$.
First examples of $\cA$-transcendental numbers are discussed in \cite{AF24,LZ25}.
Though there are two competing concepts of algebraicity in $\cA$\,---\,a na\"\i ve (`combinatorial') one from \cite{AF24} and a truly algebraic one introduced by Rosen in \cite{Ro20} (see also \cite{RTTY24}), we mainly focus our attention only on the latter (see, however, Remark~\ref{anonym} for the complete story).
Roughly speaking, an algebraic element $\bt$ in $\cA$ originates from a $C$-finite sequence $(t_m)_{m\gg1}\in\mathbb Q^\infty$, that is, from a sequence satisfying a linear recursion with constant rational coefficients.
More formally, we have the following criterion to detect the algebraicity.

\begin{prop}[{\cite[Theorem 1.1]{Ro20}}]
\label{thm:Ro}
Let $\bt=(t_p)_{p\in\cP}\in\cA$. The following conditions are equivalent.
\begin{itemize}
\item[(i)] The element $\bt$ is a finite algebraic number.
\item[(ii)] There exists a Galois extension $L/\mathbb Q$ and a map
$\phi\colon\Gal(L/\mathbb Q)\to L$ satisfying
\[
\phi(\sigma \tau\sigma^{-1})=\sigma(\phi(\tau)) \quad\text{for all}\; \sigma,\tau\in\Gal(L/\mathbb Q)
\]
such that
\[
(t_p)_{p\in\cP}=(\phi(\Frob_p)\bmod p)_{p\in\cP},
\]
where $\Frob_p$ is the Frobenius map of $L$ at the prime~$p$.
\end{itemize}
\end{prop}

As we have already indicated in \cite{LZ25}, there are somewhat natural elements in $\cA$ for which the transcendence is a consequence of existing difficult conjectures.
One such a family of potentially transcendental examples corresponds to the sequence of Frobenius traces $\ba=(a_p(E))_{p\in\cP}$ attached to an elliptic curve $E$ over~$\mathbb Q$;
we demonstrate the $\cA$-irrationality of these $\ba$ in~\cite{LZ25}.
The principal goal of this note is establish their $\cA$-transcendence unconditionally in the case of elliptic curve without complex multiplication.

\begin{theorem}
\label{thm:main}
Let $E$ be an elliptic curve defined over $\mathbb Q$, without complex multiplication.
Then the sequence of its Frobenius traces $\ba=(a_p(E))_{p\in\cP}$ is $\cA$-transcendental. 
\end{theorem}

\begin{proof}
Assume for a contradiction that $\ba$ is $\cA$-algebraic; in what follows, we abbreviate $a_p(E)$ simply as $a_p$. By part~(ii) of Proposition~\ref{thm:Ro} there exist data $L/\mathbb Q$ and $\phi\colon\Gal(L/\mathbb Q)\to L$; collect the \emph{finite} set $\{b_1,\dots,b_k\}$ of algebraic numbers given by the image of $\Gal(L/\mathbb Q)$ in~$L$.
The congruence
\[
a_p\equiv b_i\bmod p \quad\text{for some}\; i\in\{1,\dots,k\}
\]
means that for all $p$ except finitely many, the prime $p$ divides the norm $\N(a_p-b_i)$ for some $i\in \{1,\dots,k\}$.
For further use, we only need the finiteness of collection $\{b_1,\dots,b_k\}\subset\overline{\mathbb Q}$ and the latter divisibility, but not a particular choice of $L$ and~$\phi$.

Let $E_\ell$ denote the set of $\overline{\mathbb Q}$-points on the elliptic curve $E$ of order~$\ell$.
The action of the absolute Galois group $\Gal(\overline{\mathbb Q}/\mathbb Q)$ on $E_\ell$ gives a representation $\rho_\ell\colon\Gal(\overline{\mathbb Q}/\mathbb Q)\to\operatorname{Aut}(E_\ell)\cong GL_2(\mathbb Z/\ell\mathbb Z)$.
Since $E$ has no complex multiplication, Serre's open mapping theorem \cite{Se72} guarantees that, for all but finitely many \emph{primes} $\ell$, the image is the whole of $GL_2(\mathbb Z/\ell\mathbb Z)=GL_2(\mathbb F_\ell)$.
We fix any prime $\ell$ of this quality subject to the additional condition that it is relatively prime to all \emph{nonzero rational integers} in the collection $\{b_1,\dots,b_k\}$.

Now choose $X$ to be large and eliminate all primes $p\le X$ for which $a_p=0$; there are just $O(X^{3/4})$ of them, as a consequence of Kaneko's result from~\cite{Ka89} (see \cite{El91} for the details of this deduction).
Then keep only those remaining primes for which $a_p\equiv0\bmod\ell$; the latter condition means that the Frobenius at~$p$ maps  to a matrix in $GL_2(\mathbb F_\ell)$ with the zero trace\,---\,the proportion of such primes is exactly the proportion of zero-trace matrices, hence it is positive (in other words, it forms a Chebotarev set):
\[
\#\{p\le X:a_p\ne0, \,a_p\equiv0\bmod\ell\}\gg X/\log X.
\]
Here the Frobenius at $p$ is viewed as an element of $\Gal(\overline{\mathbb Q}/\mathbb Q)$ (unlike the Frobenius of a particular finite extension featured in the statement of Proposition~\ref{thm:Ro}).

Let $i$ be such that $b_i$ is nonzero and integer. For our primes the equality $a_p=b_i$ cannot happen, because $a_p\equiv0\bmod\ell$ and $(\ell,b_i)=1$ by the choice of~$\ell$.
Thus, the relation
\[
p\mid\N(a_p-b_i)  \;\text{for some}\; i\in\{1,\dots,k\}
\]
is a genuine divisibility relation in that the norm $\N(a_p-b_i)$ is not zero.
Since $p\le X$, we have $|a_p|\le 2\sqrt{X}$, so that there are $O(X^{1/2})$ ways of choosing $a_p$.
Each time this choice is performed, the nonzero integer $\N(a_p-b_i)$ has size $X^{O(1)}$, hence it has at most $O(\log X)$ prime factors. Summing up over all $a_p$ and all possible $i$, the number of primes coming out of these divisibilities is at most
\[
O(\sqrt{X}\log X).
\]
At the same time $\gg X/\log X$ such primes dividing the norms are involved.
The comparison of the sizes gives a contradiction for large~$X$.
\end{proof}

\begin{remark}
\label{anonym}
As has been pointed out to us by the anonymous referee, the na\"\i ve $\cA$-transcen\-dence \cite{AF24} of $\ba=(a_p(E))_{p\in\cP}$ can be resolved by an insignificant modification of our argument in the proof. Namely, assuming that $f(a_p)\equiv0\mod p$ for a certain polynomial $f(x)\in\mathbb Z[x]$, not necessarily irreducible but without multiple roots, we choose the set of all roots of $f(x)$ to form the collection $\{b_1,\dots,b_k\}$.
The rest of the proof proceeds without any further changes.
\end{remark}

Finally, we would like to notice that the $\cA$-transcendental examples in Theorem~\ref{thm:main} are very different in nature from the $q$-Fibonacci ones treated in \cite{AF24,LZ25}.
However, what tended to work in both situations is rough counts on how often the entries repeat on a scale up to $X$ versus how many primes are available. This is what has achieved progress.
One can expect that the argument generalises to other situations.
The elements in $\cA$ attached to Frobenius traces suggest considering the sequences $(f(p))_{p\in\cP}$ for \emph{multiplicative} functions $f\colon\mathbb Z_{>0}\to\mathbb Z$. At the same time, many of those are $\cA$-rational for trivial reasons\,---\,consider $f(n)=\sum_{d\mid n}d^k$ with $k\in\mathbb Z_{\ge0}$ as such an example; the latter choice of $f(n)$ still expects to generate $\cA$-transcendental elements $(f(p+c))_{p\in\cP}$ via a shift by a fixed $c\in\mathbb Z\setminus\{0\}$, losing the multiplicativity though.

\begin{acknowledgements}
We thank Masanobu Kaneko, Pieter Moree, Igor Shparlinski and Yuto Tsuruta for their comments on this note.
We are particularly thankful to the anonymous referee for their useful remarks and the settlement of the na\"\i ve $\cA$-transcendence in Remark~\ref{anonym}.

F.L.\ was partly supported by the 2024 ERC Synergy Project DynAMiCs.
W.Z.\ acknowledges support of the Max Planck Institute for Mathematics (Bonn, Germany) during his stay in April--June 2025.
\end{acknowledgements}


\end{document}